\documentclass[%
  twocolumn
]{mpi2015-cscpreprint}

\usepackage[american]{babel}
\usepackage{hyperref}
\usepackage[hyphenbreaks]{breakurl}

\usepackage{graphicx}

\usepackage{subfig}
\usepackage{color}

\usepackage{mymacros}
\usepackage{amssymb}
\usepackage{amsthm}

\usepackage{algorithmicx}
\usepackage{algorithm}
\usepackage{algpseudocode}

\usepackage{tikz,pgfplots} 
\usepackage{float}
\usepackage{color}
\usepackage{array}
\usepackage{dsfont}
\usepackage{tikz}
\usetikzlibrary{arrows,decorations.pathmorphing,backgrounds,fit,positioning,shapes.symbols,chains}

\newcommand{\Lbd}{\boldsymbol{\Lambda}}
\newcommand{\SO}{\texttt{SO}}
\newcommand{\Xinit}{\bx_0}
\newcommand{\Vinit}{\bv_0}
\newcommand{\bdx}{\dot{\bx}}
\newcommand{\bddx}{\ddot{\bx}}
\newcommand{\bhH}{\mathbf{\hat{H}}}
\newcommand{\bhy}{\mathbf{\hat{y}}}

\newcommand{\bSigma}{\mathbf{\Sigma}}
\newcommand{\rr}{\mathrm{r}}
\newcommand{\T}{\mathrm{T}}
\newcommand{\HH}{\mathrm{H}}
\newcommand{\rmc}{\mathrm{c}}

\newcommand{\bhM}{\mathbf{\hat{M}}}
\newcommand{\bhD}{\mathbf{\hat{D}}}
\newcommand{\bhK}{\mathbf{\hat{K}}}
\newcommand{\bhB}{\mathbf{\hat{B}}}
\newcommand{\bhC}{\mathbf{\hat{C}}}
\newcommand{\bhX}{\mathbf{\hat{X}}}
\newcommand{\bhV}{\mathbf{\hat{V}}}

\newtheorem{definition}{Definition}[section]
\newtheorem{proposition}{Proposition}[section]
\newcounter{mymac@matlab}
\newcommand{\matlab}{MATLAB%
	\ifnum\value{mymac@matlab}<1%
	\textsuperscript{\textregistered}%
	\setcounter{mymac@matlab}{1}%
	\fi%
}

\begin{document}\sloppy
  

\title{Model reduction for second-order systems with inhomogeneous initial conditions}
  
\author[$\ast,\diamond$]{Jennifer Przybilla}
\affil[$\diamond$]{Max Planck Institute for Dynamics of Complex Technical Systems, Sandtorstra{\ss}e, 39106 Magdeburg, Germany.}
\affil[$\ast$]{
  \email{przybilla@mpi-magdeburg.mpg.de}, \orcid{0000-0002-8703-8735}}
  
\author[$\dagger,\diamond$]{Igor Pontes Duff}
\affil[$\dagger$]{
  \email{pontes@mpi-magdeburg.mpg.de}, \orcid{0000-0001-6433-6142}}

\author[$\circ,\diamond$]{Peter Benner}
\affil[$\circ$]{
	\email{benner@mpi-magdeburg.mpg.de}, \orcid{0000-0003-3362-4103}}

\shorttitle{MOR for SO systems with  inhomogeneous initial conditions}
\shortauthor{J. Przybilla, I. Pontes Duff, P. Benner}
\shortdate{}


  
\abstract{%
In this paper, we consider the problem of finding surrogate models for large-scale second-order linear time-invariant systems with inhomogeneous initial conditions. 
For this class of systems, the superposition principle allows us to decompose the system behavior into three independent components.  
The first behavior corresponds to the transfer between the input and output having zero initial conditions. 
In contrast, the other two correspond to the transfer between the initial position and the initial velocity conditions having zero input, respectively. 
Based on this superposition of systems, our goal is to propose model reduction schemes allowing to preserve the second-order structure in the surrogate models. To this aim, we introduce tailored second-order Gramians for each system component and compute them numerically, solving Lyapunov equations.  As a consequence, two methodologies are proposed. The first one consists in reducing each of the components independently using a suitable balanced truncation procedure. The sum of these reduced systems provides an approximation of the original system.  This methodology allows flexibility on the order of the reduced-order model.  The second proposed methodology consists in extracting the dominant subspaces  from the sum of Gramians to build the projection matrices leading to a surrogate model.
Additionally, we discuss error bounds for the overall output approximation. Finally, the proposed methods are illustrated by means of benchmark problems.}

\novelty{This article proposes two new reduction techniques for second-order systems with inhomogeneous initial conditions.
We make use of the superposition principles to obtain three subsystems that can be evaluated separately. 
For this purpose, we derive new Gramians that describe the initial condition to output behavior. 
We discuss one algorithm that reduces the subsystems individually and one algorithm that generates one reduced system approximating the input and initial condition to output behavior. Additionally, we propose error bounds for both methods.}

\maketitle

  \section{Introduction}
  
  Second-order dynamical systems arise in many engineering applications, e.g., electrical circuits, structural dynamics,  and vibration analysis.
  In many setups, these systems are modeled by partial differential equations having second-order time-derivatives.  In order to compute the numerical simulations, spatial discretizations are needed, leading to high fidelity models.  
  However, those high fidelity models may present a high number of degrees of freedom, which are not suitable for numerical computations.  
  Consequently, model order reduction techniques are used to construct reduced-order models, that approximate the behavior of the original system.
  
  Most reduction techniques assume that the considered systems have zero initial conditions.  Consequently, these methods fail in approximating systems if they have inhomogeneous initial conditions.  Additionally, the corresponding error estimators of these methods are not applicable in this case.  This work is dedicated to finding surrogate models for second-order systems with inhomogeneous initial conditions while preserving the system structure.

  In the literature, there exist several reduction methods dedicated to systems with homogeneous initial conditions. Examples are singular value based approaches such as balanced truncation \cite{morMoo81, morTomP87, morBenOCetal17} and Hankel norm approximations \cite{morGlo84}. Additionally, there exist Krylov based methods such as the iterative rational Krylov algorithm (IRKA) \cite{morBenOCetal17, morGugSW13, morFlaBG12}, as well as, data driven methods such as the Loewner framework \cite{morMayA07}. 
  In this work, we consider second-order continuous-time dynamical systems governed by the system of differential equations 
  \begin{subequations}\label{eq:SO_sys}
  	\begin{align}
  		\bM\bddx(t) + \bD \bdx(t) + \bK\bx(t) &= \bB\bu(t), \label{eq:SO_sys_in} 
  		\\
  		\by(t) &= \bC\bx(t), \label{eq:SO_sys_out}
  		\\ 
  		\bx(0) &= \bx_0, \quad \bdx(0) = \bdx_0, 
  	\end{align}
  \end{subequations}
  where $\bM, \bD, \bK\in \R^{n \times n}$, $\bB \in \R^{n\times m}$,  $\bC\in \R^{p \times n}$, $\bx(t)\in\mathbb{R}^n$, $\bu(t)\in\mathbb{R}^m$ and $\by(t)\in\mathbb{R}^p$. 
  We assume that the position and velocity initial condition are not known  a priori.  However, they are assumed to lie in two known subspaces $\cX_0:=\myspan{\bX_0}$ and $\cV_0:= \myspan{\bV_0}$, respectively, with $\bX_0\in \R^{n\times n_{\bx_0}}$ and $\bV_0\in \R^{n\times n_{\bv_0}}$. 
  Hence,  the initial conditions can be expressed as 
  \begin{equation}\label{eq:InitCond}
  	\bx(0) = \bX_0 \bz_0,~\text{and}\quad \bdx(0) = \bV_0 \bw_0.  
  \end{equation}

  Our main goal in this work is to find low dimensional surrogate models for the system \eqref{eq:SO_sys} with inhomogeneous conditions \eqref{eq:InitCond} preserving the second-order system structure. By structure-preserving,  we mean to  determine Petrov-Galerkin projection matrices $\bW, \bV\in \R^{n \times r}$ leading to  a reduced second-order system
  \begin{align}\label{eq:SO_red}
  		\bhM\bddx_{\rr}(t) +\bhD \bdx_{\rr}(t) + \bhK\bx_{\rr}(t) &= \bhB\bu(t),\nonumber\\
  		\by_{\rr}(t) &= \bhC\bx_{\rr}(t),\\
  		\bx_{\rr}(0) &= \bhX_0 \hat{\bz}_0, \quad \bdx_{\rr}(0) = \bhV_0 \hat{\bw}_0,\nonumber
  \end{align}
  with $\bhM=\bW^{\T}\bM\bV$, $\bhD= \bW^{\T}\bD\bV$, $\bhK = \bW^{\T}\bK\bV$, $\bhB=\bW^{\T}\bB$, $ \bhC=\bC\bV$, $\bhX_0 = \bW^{\T}\bX_0$, $\bhV_0 = \bW^{\T}\bV_0$, and $\bx_{\rr}(t)\in \R^r$.

  In the literature, there exist several methods enabling model order reduction preserving the second-order structure  \cite{morFre05, morChaGVetal05}. 
  These techniques range from balanced truncation as well as balancing related model order reduction \cite{morChaLVetal06, morReiS08, morSor05} to moment matching approximations based on the Krylov subspace method \cite{morSalL06, morBeaG05}. 
  The recent work \cite{morSaaSW19} provided an extensive comparison among common second-order model reduction methods applied to a large-scale mechanical fishtail model. 
  Additionally,  \cite{morBeaG09} proposed interpolation based methods for systems possessing very general dynamical structures. More recently, the authors in \cite{morBenGP19} propose a new philosophy enabling to find the dominant reachability and observability subspaces enabling very accurate reduced-order models preserving the structure. Moreover, an extension of the Loewner framework was proposed in  \cite{morPonGB20a} for the class of Rayleigh damped systems and in \cite{morSchUBG18} for  general structured systems.

  Up to our knowledge, there is no dedicated work on system theoretical model reduction of second-order systems with inhomogeneous initial conditions. 
  For the class of first-order systems with inhomogeneous conditions, we briefly review four proposed approaches from the literature. 
  In \cite{morBauBF14}, the authors proposed to shift the state by the initial condition $\Xinit$, e.g. the new state is given as $\tilde{\bx}(t):=\bx(t)-\Xinit$. 
  That way, the initial condition is included in the input and output equation and therefore considered in the reduction process. 
  This method, however, is not straightforward applicable to second-order systems if we have a velocity initial condition and want to preserve the second-order structure.
  
  In \cite{morHeiRA11} the input $\bB\bu(t)$ is extended by the initial condition space $\bX_0$.
  More detailed, a new input matrix $\tilde{\bB}:=[\bB\;\bX_0]$ and a new input $[\bu(t)\; \bz_0]^{\T}$ are defined such that the initial condition is taken into account applying reduction methods.
  As in the previous method, this approach is not feasible if we consider velocity initial conditions in the second-order case.
  
  In \cite{morBeaGM17}, the authors' strategy is to decompose the system into a zero initial condition system and a system with initial conditions but no input.  
  The sum of the two corresponding outputs provides the original output. 
  This superposition is used to reduce these two systems, separately. 
  Extensions of the proposed methodology for the class of bilinear systems is proposed in \cite{morCaoBPetal20}  and \cite{morRedP22} based on different splittings.
  
  A recent approach \cite{morSchV20} proposes a new balanced truncation procedure based on the shift transformation on the state. This transformation is depending on design parameters allowing some flexibility and enabling the generalization of the methodologies proposed in \cite{morHeiRA11}  and \cite{morBeaGM17}.  Additionally, those parameters can be optimized, leading to accurate reduced-order models.

  In this paper, the superposition ideas in \cite{morBeaGM17} are extended to the class of second-order systems.
  For the later class, we show that, due to the superposition principle,  the original system can be decomposed into three subsystems.
  The first subsystem corresponds to the map between the input $\bu(t)$ and the output while the initial conditions are set to zero. Additionally, the second subsystem is the output resulting from the position initial condition $\bx(0)$ and the third one corresponds to the output obtained using the velocity initial condition $\dot{\bx}(0)$.
  Hence, we analyze the three corresponding subsystems separately.
  For the systems which result from the initial conditions, we develop balancing based reduction techniques based on tailored Gramians that are introduced in this paper.

  Here, two model reduction schemes are proposed. The first one consists in reducing each of the components independently using a suitable balanced truncation procedure. Hence, the sum of these reduced systems provides an approximation of the original system. As a consequence, an advantage of this approach is that the reduced dimensions and therefore the accuracies can be chosen flexibly. The second proposed methodology consists in extracting the dominant subspaces  from the sum of Gramians to build the projection matrices leading to one surrogate model.

  The rest of the paper is organized as follows. In Section~\ref{sec:BT}, we present balanced truncation for first and second-order systems.
  Afterwards, in Section \ref{sec:superpo}, we deduce a superposition of the second-order system \eqref{eq:SO_sys}.
  In Section \ref{sec:MOR}, tailored Gramians for inhomogeneous second-order systems are derived. Based on these Gramians, two model reduction schemes are proposed in Section \ref{sec:MOR_Schemes}.
  Finally, Section \ref{sec:ErrEst} provides the resulting error estimation and in Section \ref{sec:NumRes}, the methodologies are illustrated in two numerical examples.

  \section{Balanced truncation}\label{sec:BT}
  In this section, we briefly present balanced truncation method for first-order and second-order systems having zero initial conditions.
  
  \subsection{First-order systems}\label{sec:BT_FO}
  We consider the first-order dynamical system with zero initial conditions
  
  \begin{align}\label{eq:FO_sys}
  	\begin{split}
  		\boldsymbol{\cE}\dot{\bz}(t) &= \boldsymbol{\cA}\bz(t) + \boldsymbol{\cB}\bu(t),\\
  		\by(t) &= \boldsymbol{\cC}\bz(t),\\
  		\bz(0) &= 0,
  	\end{split}
  \end{align}
  with $\boldsymbol{\cA},\boldsymbol{\cE}\in\mathbb{R}^{N\times N},$ 
  $ \boldsymbol{\cB}\in\mathbb{R}^{N\times m},$ $\boldsymbol{\cC}\in\mathbb{R}^{p\times N}$, $\bz(t)\in\mathbb{R}^{N}$, $ \bu(t)\in\mathbb{R}^{m},\, \by(t)\in\mathbb{R}^{p}$.
  We assume that the system is asymptotically stable, i.e. all eigenvalues $\lambda$ of the matrix pencil $\boldsymbol{\cA}-\lambda\boldsymbol{\cE}$ fulfill $\mathrm{Re}(\lambda)<0$. 
  
  The goal of balanced truncation is to find a reduced-order model, that approximates the input-output behavior of \eqref{eq:FO_sys}.
  We recall the Laplace transform $\cL\{\bw\}$ of a function $\bw$ defined for positive values as
  \[
  \bW(s):=\cL\{\bw\}(s)=\int_{0}^{\infty}\bw(t)\exp(-st)\mathrm{d}t.
  \]
  The Laplace transform fulfills the initial condition property
  \[ \cL\{\dot{\bw}\}(s) = s\cL\{\bw\}(s) - \bw(0). \]
  Applying the Laplace transform to system \eqref{eq:FO_sys} provides 
  \[
  \bY(s) = \boldsymbol{\cC}\left( \boldsymbol{\cA} -s\boldsymbol{\cE}\right)^{-1}\boldsymbol{\cB}\bU(s),
  \]
  where $\bY$ and $\bU$ are the Laplace transforms of $\by$ and $\bu$.
  The mapping $\bH(s) := \boldsymbol{\cC}\left( \boldsymbol{\cA} -s\boldsymbol{\cE} \right)^{-1}\boldsymbol{\cB}$ is called \emph{transfer function}.
  \begin{definition}
  	The \emph{input to state mapping} $\boldsymbol{\cR}$ and the \emph{state to output mapping} $\boldsymbol{\cS}$ of system \eqref{eq:FO_sys} are
  	\[
  	\boldsymbol{\cR}(s) := \left( \boldsymbol{\cA} -s\boldsymbol{\cE}\right)^{-1}\boldsymbol{\cB},\qquad \boldsymbol{\cS}(s) : =\boldsymbol{\cC}\left( \boldsymbol{\cA} -s\boldsymbol{\cE}\right)^{-1}.
  	\]
  	The corresponding \emph{controllability} and the \emph{transformed observability Gramian} are defined as
  	\[
  	\boldsymbol{\cP} = \int_{\R} \boldsymbol{\cR}(i\omega)\boldsymbol{\cR}(-i\omega)^{\T}\mathrm{d}\omega,    \quad \boldsymbol{\cQ} = \int_{\R} \boldsymbol{\cS}(-i\omega)^{\T}\boldsymbol{\cS}(i\omega)\mathrm{d}\omega. 
  	\]
  \end{definition}
  The Gramian $\boldsymbol{\cP}$ and the transformed Gramian $\boldsymbol{\cQ}$ can be computed by solving the Lyapunov equations 
  \[
  \boldsymbol{\cA}\boldsymbol{\cP}\boldsymbol{\cE}^{\T} + \boldsymbol{\cE}\boldsymbol{\cP}\boldsymbol{\cA}^{\T} = -\boldsymbol{\cB}\boldsymbol{\cB}^{\T}, \quad \boldsymbol{\cA}^{\T}\boldsymbol{\cQ}\boldsymbol{\cE} + \boldsymbol{\cE}^{\T}\boldsymbol{\cQ}\boldsymbol{\cA} = -\boldsymbol{\cC}^{\T}\boldsymbol{\cC}.
  \]
  Small singular values of $\boldsymbol{\cP}$ and $\boldsymbol{\cE}^{\T}\boldsymbol{\cQ}\boldsymbol{\cE}$ correspond to states, that are difficult to reach and to observe.
  In order to truncate small singular values of $\boldsymbol{\cP}$ and $\boldsymbol{\cE}^{\T}\boldsymbol{\cQ}\boldsymbol{\cE}$, simultaneously, we transform the system such that the transformed Gramians $\tilde{\boldsymbol{\cP}}$, $\tilde{\boldsymbol{\cQ}}$ fulfill $\tilde{\boldsymbol{\cP}}=\tilde{\boldsymbol{\cQ}}=\boldsymbol{\Sigma} =\diag{\sigma_1,\dots,\sigma_n}$,
  where $\sigma_1,\dots,\sigma_n$ are called \emph{Hankel singular values}.
  This transformation is called balancing.
  Afterwards, we truncate the $n-r$ smallest Hankel singular values $\sigma_{r+1},\dots,\sigma_n$, $r\ll n$.
  Therefore, we consider the low-rank factors $\bR\bR^{\T}=\boldsymbol{\cP}$ and $\bS\bS^{\T}=\boldsymbol{\cQ}$ and compute the singular value decomposition 
  $
  \bS^{\T}\boldsymbol{\cE}\bR = \bU\boldsymbol{\Sigma}\bX^{\T}.
  $
  The resulting projection matrices that do both, balance and truncate, are 
  \[
  \boldsymbol{\cW} := \bS \bU_{\rr} \boldsymbol{\Sigma}_{\rr}^{-\frac{1}{2}}, \qquad \boldsymbol{\cV} := \bR \bX_{\rr}\boldsymbol{\Sigma}_{\rr}^{-\frac{1}{2}},
  \]
  where $\boldsymbol{\Sigma}_{\rr}:=\diag{\sigma_1,\dots, \sigma_r}$ and $\bU_{\rr}$ and $\bX_{\rr}$ include the $r$ leading columns of $\bU$ and $\bX$.
  The balanced and truncated system is then given by 
  \begin{align}\label{eq:red_FO_sys}
  	\begin{split}
  		\bdx_{\rr}(t) &= \boldsymbol{\cW}^{\T}\boldsymbol{\cA}\boldsymbol{\cV} \bx_{\rr}(t) + \boldsymbol{\cW}^{\T}\boldsymbol{\cB} \bu(t),\\
  		\by_{\rr}(t) &= \boldsymbol{\cC} \boldsymbol{\cV} \bx_{\rr}(t)
  	\end{split}
  \end{align}
  since 
  $
  \boldsymbol{\cW}^{\T}\boldsymbol{\cE}\boldsymbol{\cV}=\bI_r.
  $
  For more details about standard balanced truncation, see \cite{morMoo81, morBenB17}.
  
  \subsection{Second-order systems}\label{sec:BT_SO}
  Balanced truncation for second-order systems \eqref{eq:SO_sys} with zero initial conditions is presented in \cite{morChaLVetal06}.
  By applying the Laplace transform to the second-order system system \eqref{eq:SO_sys} with zero initial conditions provides the following transfer function:
  \[
  \bH_{\mathtt{SO}}(s) = \bC(s^2\bM + s\bD + \bK)^{-1}\bB.
  \]
  First, we define the input to state and the state to output mapping that result from the transfer function.
  \begin{definition}\label{def:SO_mappingsGramians}
  	The \emph{input to state mapping} $\boldsymbol{\cR}_{\emph{\text{\SO}}}$ and the \emph{state to output mapping} $\boldsymbol{\cS}_{\emph{\text{\SO}}}$ of the second-order system \eqref{eq:SO_sys} with zero initial conditions are
  	\begin{align*}
  		\boldsymbol{\cR}_{\emph{\text{\SO}}}(s) &:= \left( s^2\bM +s\bD + \bK \right)^{-1}\bB,\\
  		\boldsymbol{\cS}_{\emph{\text{\SO}}}(s) &:=\bC\left( s^2\bM +s\bD + \bK\right)^{-1}.
  	\end{align*}
  	The corresponding \emph{second-order controllability Gramian} $\bP_{\emph{\text{\SO}}}$ and \emph{observability Gramian} $\bQ_{\emph{\text{\SO}}}$ are defined by 
  	\begin{align*}
  		\bP_{\emph{\text{\SO}}}&:=\int_{\R} \boldsymbol{\cR}_{\emph{\text{\SO}}}(iw)\boldsymbol{\cR}_{\emph{\text{\SO}}}(-iw)^{\T}\mathrm{d}w,\\
  		\bQ_{\emph{\text{\SO}}}&:=\int_{\R} \boldsymbol{\cS}_{\emph{\text{\SO}}}(-iw)^{\T}\boldsymbol{\cS}_{\emph{\text{\SO}}}(iw)\mathrm{d}w.
  	\end{align*}
  \end{definition}
  In order to reduce the second-order system \eqref{eq:SO_sys} with zero initial conditions, we transform it to a first-order system \eqref{eq:FO_sys} by setting
  \begin{align*}
  	\boldsymbol{\cE}&:=\begin{bmatrix}
  		\bI & 0 \\
  		0 & \bM
  	\end{bmatrix}, \quad
  	\boldsymbol{\cA}:=\begin{bmatrix}
  		0 & \bI \\
  		-\bK & -\bD
  	\end{bmatrix},\quad \boldsymbol{\cB}:=\begin{bmatrix}
  	0 \\
  	\bB
  \end{bmatrix}, \\ \boldsymbol{\cC}&:= \begin{bmatrix}
  \bC & 0
\end{bmatrix}.
\end{align*}
The first-order system is then equivalent to the second-order system and the corresponding transfer function is given by
\[
\bH_{{\text{\SO}}}(s) = \boldsymbol{\cC}\left( \boldsymbol{\cA} -s\boldsymbol{\cE} \right)^{-1}\boldsymbol{\cB} = \bC(s^2\bM + s\bD + \bK)^{-1}\bB.
\]
Note that there exist several first-order representation that are equivalent to system \eqref{eq:SO_sys}.
We determine the second-order controllability Gramian of system \eqref{eq:SO_sys} with zero initial conditions by considering the Gramian of the first-order system \eqref{eq:FO_sys} as described in the following proposition.
\begin{proposition}
The second-order controllability Gramian 
$\bP_{\emph{\text{\SO}}}$ of System \eqref{eq:SO_sys} with zero initial conditions is equal to the upper left block $\bP_1$ of the first-order controllability Gramian
\begin{align*}
\boldsymbol{\cP} &= \begin{bmatrix}
\bP_1 & \bP_2 \\
\bP_2^{\T} & \bP_3
\end{bmatrix} = \int_{\R}(\boldsymbol{\cA}-iw\boldsymbol{\cE})^{-1}\boldsymbol{\cB}\boldsymbol{\cB}^{\T}(\boldsymbol{\cA}+iw\boldsymbol{\cE})^{-\T}\mathrm{d}w\\
&= \int_{\R}\begin{bmatrix}
				-iw \bI & \bI\\
				-\bK & -\bD -iw\bM
			\end{bmatrix}^{-1}\begin{bmatrix}
			0\\
			\bB
		\end{bmatrix}
		\boldsymbol{\cB}^{\T}(\boldsymbol{\cA}+iw\boldsymbol{\cE})^{-\T}\mathrm{d}w.
\end{align*}
\end{proposition}
\begin{proof}
Applying the Schur complement provides  that $\bP_1$ is given by
\begin{multline*}
\bP_1 = \int_{\R}((iw)^2\bM + iw\bD + \bK)^{-1}\bB \\
\cdot\bB^{\T}((iw)^2\bM - iw\bD + \bK)^{-\T}\mathrm{d}w.
\end{multline*}
\end{proof}

The matrix $\bP_{\text{\SO}}$ is called position controllability Gramian.
We observe that $\bP_{\text{\SO}}$ encodes the important subspaces of the map between the input to state of the homogeneous second-order system \eqref{eq:SO_sys}.
Hence, $\bP_{\text{\SO}}$ spans the controllability space and is used to apply balanced truncation in the second-order case.

The same argumentation is used to extract the state to output mapping space from the first-order observability Gramian $\boldsymbol{\cQ}= \begin{bmatrix}
\bQ_1 & \bQ_2\\\
\bQ_2^{\T} & \bQ_3
\end{bmatrix}$.
The second-order observability Gramian $\bQ_{\text{\SO}}$ presented in Definition \eqref{def:SO_mappingsGramians} is equal to the velocity observability Gramian $\bQ_3$.

As in the first-order case we use the low-rank factors $\bR_1$ and $\bS_3$ with $\bP_1 = \bR_1\bR_1^{\T}$ and $\bQ_3 = \bS_3\bS_3^{\T}$ and compute the singular value decomposition 
$
\bS_3^{\T}\bR_1 = \bU\boldmath{\Sigma}\bX^{\T}.
$
The resulting balancing and truncating projection matrices are 
\begin{equation}\label{eq:SOprojection_matrices}
	\bW := \bS_3 \bU_{\rr} \boldsymbol{\Sigma}_{\rr}^{-\frac{1}{2}}, \qquad \bV := \bR_1 \bX_{\rr}\boldsymbol{\Sigma}_{\rr}^{-\frac{1}{2}},
\end{equation}
where $\boldsymbol{\Sigma}_{\rr}$ is the diagonal matrix containing the $r$ largest singular values of $\boldsymbol{\Sigma}$. 
Moreover, $\bU_{\rr}$ and $\bX_{\rr}$ include the $r$ leading columns of $\bU$ and $\bX$.
Projecting by $\bW$ and $\bV$ provides the reduced system \eqref{eq:SO_red} with zero initial conditions.

\section{Superposition principle for second-order systems}\label{sec:superpo}
This section aims at  decomposing the original system behavior of the second-order system \eqref{eq:SO_sys} in simpler subsystems. This system decomposition will be the inspiration of the proposed model reduction schemes.

By applying the Laplace transform to equation \eqref{eq:SO_sys_in} we obtain
\begin{align*}
	\bB \bU(s)
	&= \bM \cL\{\ddot{\bx}\}(s) + \bD \cL\{\dot{\bx}\}(s) + \bK\cL\{\bx\}(s)\\
	&= \bM(s^2\bX(s) - s\bx(0) - \dot{\bx}(0)) \\
	&\hspace*{60pt}+ \bD (s\bX(s)- \bx(0)) + \bK\bX(s)
\end{align*}
where $\bX$ is the Laplace transform of $\bx$ and $\bU$ the Laplace transform of $\bu$.
Hence, it holds that
\[(s^2\bM+s\bD+\bK)\bX(s) = \bB\bU(s) +\bD\bx(0) +s\bM\bx(0)+\bM\dot{\bx}(0). \]
Applying the Laplace transform to equation \eqref{eq:SO_sys_out} and defining $\bY$ as the Laplace transform of $\by$ provides 
\begin{multline*}
\bY(s) = \bC\Lbd(s)\bB\bU(s) + \bC\Lbd(s)(s\bM+\bD)\bX_0\bz_0 \\+ \bC\Lbd(s)\bM\bV_0\bw_0
\end{multline*}
for $\Lbd(s) :=  (s^2\bM+s\bD+\bK)^{-1}$.
We observe that the output is a superposition of the input to output mapping, the position initial condition to output mapping and the velocity initial condition to output mapping.
As a consequence, the global input-output behavior is given by 
\begin{equation*}
	\bY(s) =\bC\bX(s) = \bH_{\SO}(s)\bU(s) +\bH_{\Xinit}(s)\bz_0 +  \bH_{\Vinit}(s)\bw_0 
\end{equation*}
where 
\begin{equation*}
	\begin{array}{rl}
		\bH_{\SO}(s) :=  \bC\Lbd(s)\bB, &	
		\bH_{\Xinit}(s) := \bC\Lbd(s)(\bD +s\bM)\bX_0, 
		\\ 
		\text{and}&  
		\bH_{\Vinit}(s) :=\bC\Lbd(s)\bM\bV_0.
	\end{array}
\end{equation*}
Up to now, we saw that three independent transfer functions characterize the inhomogeneous behavior of the second-order realization \eqref{eq:SO_sys}.

The transfer function $\bH_{\SO}(s) = \bC\Lbd(s)\bB$  corresponds to the input and output map without initial conditions. Hence, it is associated with the following realization
\begin{align}\label{eq:SO_homo}
	\begin{split}
		\bM\ddot{\bx}_{\SO}(t) + \bD \dot{\bx}_{\SO}(t) + \bK\bx_{\SO}(t) &= \bB\bu(t), \\
		\by_{\SO}(t) &= \bC\bx_{\SO}(t),\\
		\bx_{\SO}(0)&= 0, \quad  \bdx_{\SO}(0) =0.
	\end{split}
\end{align}
The transfer function $\bH_{\Xinit}(s) = \bC\Lbd(s)(\bD +s\bM)\bX_0$ corresponds to the transfer between the initial position condition and the output. Hence, the following realization is associated to it:
\begin{align}\label{eq:SO_state_init}
		\bM\ddot{\bx}_{\Xinit}(t) + \bD \dot{\bx}_{\Xinit}(t) + \bK\bx_{\Xinit}(t) &= 0,\nonumber \\
		\by_{\Xinit}(t) &= \bC\bx_{\Xinit}(t),\\
		\bx_{\Xinit}(0)&= \bX_0\bz_0, \quad  \bdx_{\Xinit}(0) =0.\nonumber
\end{align}
Finally, we write the realization for $\bH_{\Vinit}(s) =\bC\Lbd(s)\bM\bV_0$.
This transfer function corresponds to the transfer between the initial velocity condition and the output. The following realization is associated to it:
\begin{align}\label{eq:SO_velo_init}
		\bM\ddot{\bx}_{\Vinit}(t) + \bD \dot{\bx}_{\Vinit}(t) + \bK\bx_{\Vinit}(t) &= 0,\nonumber \\
		\by_{\Vinit}(t) &= \bC\bx_{\Vinit}(t),\\
		\bx_{\Vinit}(0)&= 0, \quad  \bdx_{\Vinit}(0) = \bV_0\bw_0.\nonumber
\end{align}
To summarize, we have seen that the output of the inhomogeneous second-order system in \eqref{eq:SO_sys} can be decomposed as
\[ \by(t) = \by_{\SO}(t) + \by_{\Xinit}(t) + \by_{\Vinit}(t) \]
governed by the transfer functions $\bH_{\SO}$, $\bH_{\Xinit}$ and $\bH_{\Vinit}$.
Figure \ref{fig:Sigma_1sys} sketches the input and initial conditions to output behavior of the original second-order system \eqref{eq:SO_sys}, while Figure \ref{fig:Sigma_3sys} draws the superposition of the original system into three independent systems.
Therefore, the sum of the separately computed outputs leads to the same output as the original system \eqref{eq:SO_sys}.
\begin{figure}\label{fig:Sigma_1sysTotal} 
	\centering
	\subfloat[Original system]{%
		\resizebox*{6cm}{!}{\begin{tikzpicture}
			\fill[blue!20!white] (0,0) rectangle (2,1.5);
			\node (Sigma) at (1,0.75) {\Large{$\Sigma$}};
			\draw[->, black!90, very thick]  (-1.5,1.4) -- (0,1.4);
			\node (w) at (-0.75,1.55) {$\bu$};
			\draw[->, black!90, very thick]  (-1.5,0.75) -- (0,0.75);
			\node (w) at (-0.75,0.9) {$\bx_0$};
			\draw[->, black!90, very thick]  (-1.5,0.1) -- (0,0.1);
			\node (u) at (-0.75,0.25) {$\dot{\bx}_0$};
			\draw[->, black!90, very thick]  (2,0.75) -- (3.5,0.75);
			\node (y) at (2.75,0.92) {$\by$};
			\end{tikzpicture}} 
		\label{fig:Sigma_1sys}}\vspace{20pt}
	\subfloat[Superposition of the system]{%
		\resizebox*{6cm}{!}{\begin{tikzpicture}
			\fill[blue!20!white] (0,2) rectangle (1,2.75);
			\node (Sigma) at (0.5,2.36) {\large{$\Sigma_{\SO}$}};
			\fill[blue!20!white] (0,1) rectangle (1,1.75);
			\node (Sigma) at (0.5,1.36) {\large{$\Sigma_{\Xinit}$}};
			\fill[blue!20!white] (0,0) rectangle (1,0.75);
			\node (Sigma) at (0.5,0.36) {\large{$\Sigma_{\dot{\bx}_0}$}};
			\draw[->, black!90, very thick]  (-1.5,2.36) -- (0,2.36);
			\node (w) at (-0.75,2.5) {$\bu$};
			
			\draw[->, black!90, very thick]  (-1.5,1.36) -- (0,1.36);
			\node (w) at (-0.75,1.5) {\small $\bx_0$};
			\draw[->, black!90, very thick]  (-1.5,0.36) -- (0,0.36);
			\node (u) at (-0.75,0.5) {\small $\dot{\bx}_0$};
			\draw[->, black!90, very thick]  (2.2,1.36) -- (3,1.36);
			
			\node (y) at (2.5,1.52) {\small{$\by$}};
			\draw[black!90, very thick]  (1,0.36) -- (2,0.36);
			\node (w) at (1.5,2.5) {\small{$\by_{\SO}$}};
			\draw[black!90, very thick]  (1,1.36) -- (1.8,1.36);
			\node (w) at (1.5,1.5) {\small{$\by_{\bx_0}$}};
			\draw[black!90, very thick]  (1,2.36) -- (2,2.36);
			\node (u) at (1.5,0.5) {\small{$\by_{\dot{\bx}_0}$}};
			\draw[black!90, very thick]  (2,1.56) -- (2,2.36);
			\draw[black!90, very thick]  (2,0.36) -- (2,1.15);
			\draw[black!90, very thick] (2,1.36) circle (6pt);
			\node (Sig) at (2,1.36) {\Large{$+$}};
			\end{tikzpicture}}
		\label{fig:Sigma_3sys}}
	\caption{System structures}
\end{figure}
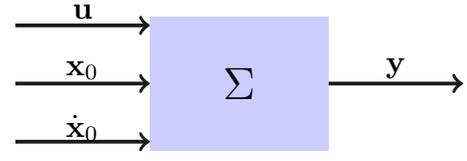
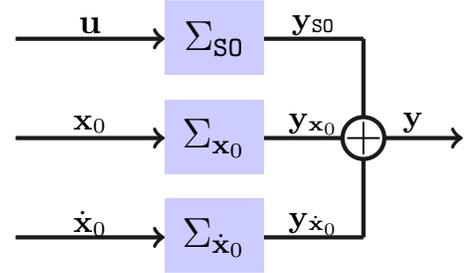

\section{Gramians of inhomogeneous second-order systems}\label{sec:MOR}

In order to derive the proposed model reduction schemes, we analyze separately the three subsystems and we introduce tailored Gramians for each one of them.

Notice that subsystem \eqref{eq:SO_homo} corresponds to a second-order realization with homogeneous initial conditions. Hence,  the controllability and observability Gramians $\bP_{\SO}$ and $\bQ_{\SO}$ presented in Definition \ref{def:SO_mappingsGramians} can be used to characterize its dominant subspaces.

However, subsystems \eqref{eq:SO_state_init} and \eqref{eq:SO_velo_init} have a different structure, and hence, tailored Gramians are required.
In Section \ref{sec:SO_Gramian2} and \ref{sec:SO_Gramian1}, we propose tailored Gramians for these subsystems.
Afterwards, in Section \ref{sec:MOR_Schemes}, we propose two different MOR schemes based on these Gramians.

\subsection{Gramians of $\bH_{\Xinit}$ }\label{sec:SO_Gramian2}

Considering the transfer function $\bH_{\Xinit}(s)$ of system \eqref{eq:SO_state_init} more detailed shows that the input to state mapping differs from the structure in Definition \ref{def:SO_mappingsGramians}.
The state to output mapping, however, is the same.
Hence, we define the input to state mapping and the corresponding second-order Gramian.
\begin{definition}\label{def:in_stat_map_Xinit}
	The \emph{input to state mapping} $\boldsymbol{\cR}_{\Xinit}$ and the corresponding \emph{controllability Gramian} $\bP_{\Xinit}$ of the second-order system \eqref{eq:SO_state_init} are
	\begin{align*}
		&\boldsymbol{\cR}_{\Xinit}(iw):=\bC(s^2\bM + s\bD + \bK)^{-1}(\bD + s\bM)\bX_0,\\
		&\bP_{\Xinit} := \int_{\R} \boldsymbol{\cR}_{\Xinit}(iw)\boldsymbol{\cR}_{\Xinit}(-iw)^{\T}\mathrm{d}w.
	\end{align*}
\end{definition}
\begin{proposition}\label{prop:in_stat_map_Xinit}
	The second-order controllability Gramian $\bP_{\Xinit}$ of System \eqref{eq:SO_state_init} described in Definition \ref{def:in_stat_map_Xinit} is the upper left matrix $\bP_1$ of
	\begin{align*}
		\boldsymbol{\cP} &= \begin{bmatrix}
			\bP_1 & \bP_2 \\
			\bP_2^{\T} & \bP_3
		\end{bmatrix}\\
		&= \int_{\R}\underbrace{(\boldsymbol{\cA}-iw\boldsymbol{\cE})^{-1}}_{:= \boldsymbol{\Gamma}(i w)}\begin{bmatrix}
			\bX_0\\
			0
		\end{bmatrix}\begin{bmatrix}
		\bX_0^{\T} & 0
	\end{bmatrix}(\boldsymbol{\cA}+iw\boldsymbol{\cE})^{-\T}\mathrm{d}w.
\end{align*}
\end{proposition}
\begin{proof}
	Applying the Schur complement to $\boldsymbol{\Gamma}(iw)$ provides that its upper left block is 
	$
	-((iw)^2\bM+iw\bD+\bK)^{-1}(iw\bM+\bD)
	$	
	and hence it holds that
	\begin{align*}
		\bP_1 &= \int_{\mathbb{R}} ((iw)^2\bM +iw\bD + \bK)^{-1}(iw\bM+\bD)\bX_0\\
		&\qquad \cdot\bX_0^{\T}(-iw\bM+\bD)^{\T}((iw)^2\bM -iw\bD + \bK)^{-\T}\mathrm{d}w\\
		&=\int_{\mathbb{R}}\boldsymbol{\cR}_{\Xinit}(iw)\boldsymbol{\cR}_{\Xinit}(-iw)^{\T}\mathrm{d}w.
	\end{align*}
\end{proof}

Proposition \ref{prop:in_stat_map_Xinit} shows that the second-order controllability Gramian $\bP_{\Xinit}$ of System \eqref{eq:SO_state_init} is given by the upper left part $\bP_1$ of the controllability Gramian $\boldsymbol{\cP}$ of the first-order system \eqref{eq:FO_sys} with $\boldsymbol{\cB}:=\footnotesize\left[\hspace*{-3pt} \begin{array}{c}
\bX_0\\
0 \end{array}\hspace*{-3pt} \right]\normalsize $. Moreover, the second-order observability Gramian $\bQ_{\Xinit}$ is equal to $\bQ_{\SO}$ since the state to output mapping $\boldsymbol{\cS}_{\Xinit}(s) : =\bC\left( s^2\bM +s\bD + \bK\right)^{-1}$ that is used to derive the observability Gramian is equal to the state to output mapping $\boldsymbol{\cS}_{\SO}(s)$ from Definition \ref{def:SO_mappingsGramians} for the homogeneous system case.

\subsection{Gramians of $\bH_{\Vinit}$}\label{sec:SO_Gramian1}
In order to apply balanced truncation to system \eqref{eq:SO_velo_init}, we define the corresponding input to state mapping.
As in the previous section, the state to output mapping is the same as for system \eqref{eq:SO_homo}.
\begin{definition}\label{def:in_velo_map_Vinit}
	The \emph{input to state mapping} $\boldsymbol{\cR}_{\Vinit}$ and the corresponding \emph{controllability Gramian} $\bP_{\Vinit}$ of the second-order system \eqref{eq:SO_velo_init} are
	\begin{align*}
		&\boldsymbol{\cR}_{\Vinit}(s) := \left(s^2\bM + s\bD + \bK \right)^{-1}\bM\bV_0,\\
		&\bP_{\Vinit} := \int_{\R} \boldsymbol{\cR}_{\Vinit}(iw)\boldsymbol{\cR}_{\Vinit}(-iw)^{\T}\mathrm{d}w.
	\end{align*}
\end{definition}
We note that the input to state mapping $\boldsymbol{\cR}_{\Vinit}(s)$ and hence the second-order controllability Gramian $\bP_{\Vinit}$ are of the same structure as in the homogeneous case, presented in Definition \ref{def:SO_mappingsGramians}.
\begin{proposition}
	The second-order controllability Gramian $\bP_{\Vinit}$ of System \eqref{eq:SO_velo_init} described in Definition \ref{def:in_velo_map_Vinit} is the upper left matrix $\bP_1$ of
	\begin{align*}
		\boldsymbol{\cP} &= \begin{bmatrix}
			\bP_1 & \bP_2 \\
			\bP_2^{\T} & \bP_3
		\end{bmatrix}\\
		&= \int_{\R}\underbrace{(\boldsymbol{\cA}-iw\boldsymbol{\cE})^{-1}}_{:= \boldsymbol{\Gamma}(i w)}\begin{bmatrix}
			0\\
			\bM\bV_0
		\end{bmatrix}\\
		&\hspace{90pt}\cdot\begin{bmatrix}
		0 & \bV_0^{\T}\bM^{\T}
	\end{bmatrix}(\boldsymbol{\cA}+iw\boldsymbol{\cE})^{-\T}\mathrm{d}w.
\end{align*}
\end{proposition}
\begin{proof}
	Applying the Schur complement to $\boldsymbol{\Gamma}(iw)$ provides that its upper right part is 
	$
	-((iw)^2\bM -iw\bD + \bK)^{-1}
	$	and hence
	\begin{align*}
		\bP_1 &= \int_{\mathbb{R}} ((iw)^2\bM +iw\bD + \bK)^{-1}\bM\bV_0\\
		&\qquad\qquad\qquad\cdot\bV_0^{\T}\bM^{\T}((iw)^2\bM -iw\bD + \bK)^{-\T}\mathrm{d}w\\
		&=\int_{\mathbb{R}}\boldsymbol{\cR}_{\Vinit}(iw)\boldsymbol{\cR}_{\Vinit}(-iw)^{\T}\mathrm{d}w.
	\end{align*}
\end{proof}

Again the second-order observability Gramian $\bQ_{\Vinit}$ is equal to the one of the homogeneous second-order system $\bQ_{\SO}$ since their state to output mappings coincide.

\section{Model reduction schemes}\label{sec:MOR_Schemes}
In this section, we present two model reduction schemes for the class of systems in \eqref{eq:SO_sys}. The procedures use the tailored Gramians presented in Section \ref{sec:MOR} and construct second-order reduced-order models via balanced truncation as presented in Subsection \ref{sec:BT_SO}.

\begin{algorithm}[tb]
	\caption{BT method for inhomogeneous second-order systems by superposition.}
	\label{algo:SplitProj}
	\begin{algorithmic}[1]
		\Require{The original matrices $\bM$, $\bK$, $\bD$, $\bB$, $\bC$, $\bX_0$, $\bV_0$ and the orders $r_{\star}$, where $\star$ is $\small{\SO}$, $\Xinit,$ or $ \Vinit$ and describes the systems \eqref{eq:SO_homo}, \eqref{eq:SO_state_init} or \eqref{eq:SO_velo_init}.}
		\Ensure{The reduced matrices $\bhM_{\star}$, $\bhK_{\star}$, $\bhD_{\star}$, $\bhB_{\star}$, $\bhC_{\star}$, $\bhX_0$, $\bhV_0$.}
		\State{Compute low-rank factors of the Gramians $\bP_{\star}\approx\bR_{\star}\bR_{\star}^{\T}$ and $\bQ \approx \bS \bS^{\T}$ from Definition \ref{def:SO_mappingsGramians}, \ref{def:in_stat_map_Xinit} and \ref{def:in_velo_map_Vinit}.}
		\State Perform the SVD of $\bS^{\T}\bR_{\star}$, and decompose as \[ \bS^{\T}\bR_{\star}
		= \begin{bmatrix}
		\bU_{\star}^{(1)} & \bU_{\star}^{(2)}
		\end{bmatrix}\diag{\bSigma_{\star}^{(1)}, \bSigma_{\star}^{(2)}}\begin{bmatrix}
		\bX_{\star}^{(1)} & \bX_{\star}^{(2)}
		\end{bmatrix}^{\T},\] 
		with $ \bSigma_{\star}^{(1)} \in \R^{r_{\star}\times r_{\star}}$.
		\State Construct the projection matrices \[ \bW_{\star} = \bS\bU_{\star}^{(1)}(\bSigma_{\star}^{(1)})^{-\frac{1}{2}}~\, \text{and}~\,\bV_{\star} = \bR_{\star}\bX_{\star}^{(1)}(\bSigma_{\star}^{(1)})^{-\frac{1}{2}}.\]
		\State Construct reduced matrices
		\begin{align*}
			\bhM_{\star} &= \bW_{\star}^{\T} \bM_{\star} \bV_{\star},\, \bhD_{\star} = \bW_{\star}^{\T} \bD_{\star} \bV_{\star},\, \bhK_{\star} = \bW_{\star}^{\T} \bK_{\star} \bV_{\star} \\ \bhB_{\star} &= \bW_{\star}^{\T} \bB_{\star}, \, \bhC_{\star} = \bC_{\star} \bV_{\star},\\
			\bhX_0 &= \bW_{\Xinit}^{\T}\bX_0, \, \bhV_0 = \bW_{\bv_0}^{\T}\bV_0.
		\end{align*}
	\end{algorithmic}
\end{algorithm}

\subsection{Method 1: Reducing each subsystem}\label{sec:sep_MOR}
The first method we propose utilizes the superposition properties to reduced the subsystems presented in Section \ref{sec:superpo} separately based on the Gramians presented in Definition \ref{def:SO_mappingsGramians} and in Section \ref{sec:MOR}.

For the homogeneous subsystem \eqref{eq:SO_homo}, we aim to apply the reduction procedure from Subsection \ref{sec:BT_SO} using the Gramians from Definition \ref{def:SO_mappingsGramians} to derive a reduced-order system with the transfer function 
\[\bhH_{\SO}(s)=\bC\bV_{\SO}\left(\bW_{\SO}^{\T}(s^2\bM + s\bD + \bK)\bV_{\SO}\right)^{-1}\bW_{\SO}^{\T}\bB,\]
where $\bW_{\SO}$ and $\bV_{\SO}$ are the corresponding projection matrices given in \eqref{eq:SOprojection_matrices}.

For the subsystem \eqref{eq:SO_state_init} describing the system behavior that results from the state initial condition, we build the corresponding balanced truncation projection matrices $\bW_{\Xinit}$ and $\bV_{\Xinit}$ as in Equation \eqref{eq:SOprojection_matrices} based on the Gramians presented in Subsection \ref{sec:SO_Gramian2}.
We reduce system \eqref{eq:SO_state_init} accordingly and obtain the reduced position initial condition transfer function
\begin{multline*}
	\bhH_{\Xinit}(s)=\bC\bV_{\Xinit}\left(\bW_{\Xinit}^{\T}(s^2\bM + s\bD + \bK)\bV_{\Xinit}\right)^{-1}\\ 
	\cdot \bW_{\Xinit}^{\T}(\bD + s\bM)\bV_{\Xinit}\bW_{\Xinit}^{\T}\bX_0.
\end{multline*}

Applying second-order balanced truncation to system \eqref{eq:SO_velo_init} using the second-order Gramians from Subsection \ref{sec:SO_Gramian1} provides the projection matrices $\bW_{\Vinit}$ and $\bV_{\Vinit}$ from Equation \eqref{eq:SOprojection_matrices}.
Reducing the system accordingly generates the corresponding reduced transfer function 
\begin{multline*}
\bhH_{\Vinit}(s)=\bC\bV_{\Vinit}\left(\bW_{\Vinit}^{\T}(s^2\bM + s\bD + \bK)\bV_{\Vinit}\right)^{-1}\\
\cdot\bW_{\Vinit}^{\T}\bM\bV_{\Vinit}\bW_{\Vinit}^{\T}\bV_0
\end{multline*}
that describes the velocity initial condition to output behavior.

Summarizing, we apply balanced truncation to the three systems to generate the corresponding reduced transfer functions  $\bhH_{\SO}$, $\bhH_{\Xinit}$ and $\bhH_{\Vinit}$ associated with the outputs  $\bhy_{\SO}(t), \bhy_{\Xinit}(t)$, and $\bhy_{\Vinit}(t)$, respectively, such that the overall behavior \[\bhy =\bhy_{\SO}(t) + \bhy_{\Xinit}(t) + \bhy_{\Vinit}(t)\] approximates the original output $\by(t)$.  The detailed reduction procedure for each subsystem is given in Algorithm \ref{algo:SplitProj}.

\subsection{Method 2: Combined Gramians}\label{sec:combGram}
We have discussed the approach where we use separated projections for each subsystem.
However, for some applications it might be advantageous to have only one projection that reduces the original system including the initial conditions at once.
That means that we need to determine a projection based on a controllability space which corresponds to the input and the initial conditions.
This controllability space is spanned by the columns of the sum of the controllability Gramians introduced in the previous sections
\begin{align}\label{eq:Gramian_combined}
	\begin{split}
		\bP_{\rmc} &= \bP_{\SO} +\bP_{\Xinit} + \bP_{\Vinit} =\begin{bmatrix}
			\bR_{\SO} & \bR_{\Xinit} & \bR_{\Vinit}
		\end{bmatrix}\begin{bmatrix}
		\bR_{\SO}^{\mathrm{T}} \\ \bR_{\Xinit}^{\mathrm{T}} \\ \bR_{\Vinit}^{\mathrm{T}}
	\end{bmatrix},\\
	\bQ_{\rmc} & =\bQ_{\SO} =\bQ_{\Xinit} =\bQ_{\Vinit} = \bS_{\SO}\bS_{\SO}^{\mathrm{T}}
\end{split}
\end{align}
where $\bR_{\SO}^{\mathrm{T}}$, $\bR_{\Xinit}^{\mathrm{T}}$, $\bR_{\Vinit}^{\mathrm{T}}$ and $\bS_{\SO}$ are the corresponding low-rank factors of $\bP_{\SO}^{\mathrm{T}}$, $\bP_{\Xinit}^{\mathrm{T}}$, $\bP_{\Vinit}^{\mathrm{T}}$ and $\bQ_{\SO}$.
Applying balanced truncation for second-order systems based on the low-rank factors of the combined Gramains $\bP_{\rmc}$ and $\bQ_{\rmc}$ from equation \eqref{eq:Gramian_combined} results in a reduced-order system that takes into account the input to state and the initial conditions to state mappings.

Another approach that results in the same controllability Gramian $\bP_{\rmc}$ and therefore the same reduced-order system would be a modification of the method presented in \cite{morHeiRA11} for second-order systems.
Therefore, we consider the homogeneous first order system \eqref{eq:FO_sys} with the input matrix
\begin{equation}\label{eq:combB}
	\boldsymbol{\cB} = \begin{bmatrix}
		0 & \bX_0 & 0 \\
		\bB & 0 & \bM\bV_0
	\end{bmatrix}.
\end{equation}
\begin{proposition}
	The position controllability Gramian of the first order system \eqref{eq:FO_sys} with $\boldsymbol{\cB}$ as in equation \eqref{eq:combB} is equal to the Gramian $\bP_{\rmc}$ in equation \eqref{eq:Gramian_combined}.
\end{proposition}
\begin{proof}
	The position controllability Gramian of system \eqref{eq:SO_sys} is described by the upper left part $\bP_1$ of
	\begin{align*}
		\boldsymbol{\cP} &= \begin{bmatrix}
			\bP_1 & \bP_2 \\
			\bP_2^{\T} & \bP_3
		\end{bmatrix}\\
		&=\int_{\R}\underbrace{(\boldsymbol{\cA}-iw\boldsymbol{\cE})^{-1}}_{:= \boldsymbol{\Gamma}(i w)}\begin{bmatrix}
			0 &\bX_0 &0\\
			\bB &0 & \bM\bV_0
		\end{bmatrix}\boldsymbol{\cB}^{\T}(\boldsymbol{\cA}+iw\boldsymbol{\cE})^{-\T}\mathrm{d}w.
	\end{align*}
	Applying the Schur complement to $\boldsymbol{\Gamma}(iw)$ provides that the upper left block is $-((iw)^2\bM +iw\bD + \bK)^{-1}(iw\bM + \bD)$ and the upper right block is $-((iw)^2\bM +iw\bD + \bK)^{-1}$.
	It follows that 
	\begin{align*}
		\bP_1 &= \int_{\R}((iw)^2\bM +iw\bD + \bK)^{-1}\bB\\
		&\hspace*{95pt}\cdot\bB^{\T}((iw)^2\bM -iw\bD + \bK)^{-\T}\\
		&+\int_{\R}((iw)^2\bM +iw\bD + \bK)^{-1}(iw\bM + \bD)\bX_0\\
		&\hspace{25pt}\cdot\bX_0^{\T}(-iw\bM + \bD)^{-\T}((iw)^2\bM -iw\bD + \bK)^{-\T}\\
		&+\int_{\R}((iw)^2\bM +iw\bD + \bK)^{-1}\bM\bV_0\\
		&\hspace{79pt}\cdot\bV_0^{\T}\bM^{\T}((iw)^2\bM -iw\bD + \bK)^{-\T}\\
		&= \bP_{\SO} +\bP_{\Xinit} + \bP_{\Vinit}.
	\end{align*}
\end{proof}

The final method is presented in Algorithm \ref{algo:CombProj} and generates the reduced transfer function 
\[\bhH(s)=\bC\bV_{\rmc}\left(\bW_{\rmc}^{\T}(s^2\bM + s\bD + \bK)\bV_{\rmc}\right)^{-1}\bW_{\rmc}^{\T}\bB.\]
The advantage of this approach is the fact that we obtain only one second-order reduced-order model approximating the behavior of the original system. The disadvantage of this method is the inflexibility of the different controllability space dimensions.
It follows, that the combined Gramian leads possibly to reduced dimensions significantly larger than the dimensions of the separately reduced systems to reach the same approximation quality.

\begin{algorithm}[tb]
	\caption{BT method for inhomogeneous second-order systems by combined Gramians.}
	\label{algo:CombProj}
	\begin{algorithmic}[1]
		\Require{The original matrices $\bM$, $\bK$, $\bD$, $\bB$, $\bC$, $\bX_0$, $\bV_0$ and the order $r$.}
		\Ensure{The reduced matrices $\bhM$, $\bhK$, $\bhD$, $\bhB$, $\bhC$, $\bhX_0$, $\bhV_0$}.
		\State Build the input matrix 
		\[
		\boldsymbol{\cB} = \begin{bmatrix}
		0 & \bX_0 & 0 \\
		\bB & 0 & \bM\bV_0
		\end{bmatrix}.
		\]
		\State{Compute low factors of Gramians $\bP\approx\bR\bR^{\T}$ from \eqref{eq:Gramian_combined} and $\bQ \approx \bS \bS^{\T}$ from Definition \ref{def:SO_mappingsGramians}.}
		\State Perform the SVD of $\bS^{\T}\bR$, and decompose as \[ \bS^{\T}\bR
		= \begin{bmatrix}
		\bU^{(1)} & \bU^{(2)}
		\end{bmatrix}\diag{\bSigma^{(1)}, \bSigma^{(2)}}\begin{bmatrix}
		\bX^{(1)} & \bX^{(2)}
		\end{bmatrix}^{\T},\] 
		with $ \bSigma^{(1)} \in \R^{r\times r}$.
		\State Construct the projection matrices \[ \bW_{\rmc} = \bS\bU^{(1)}(\bSigma^{(1)})^{-\frac{1}{2}}~\, \text{and}~\,\bV_{\rmc} = \bR\bX^{(1)}(\bSigma^{(1)})^{-\frac{1}{2}}.\]
		\State Construct reduced matrices
		\begin{align*}
			\bhM &= \bW_{\rmc}^{\T} \bM \bV_{\rmc},\, \bhD = \bW_{\rmc}^{\T} \bD \bV_{\rmc},\, \bhK = \bW_{\rmc}^{\T} \bK \bV_{\rmc} \\ \bhB &= \bW_{\rmc}^{\T} \bB, \, \bhC = \bC \bV_{\rmc},\,\bhX_0 = \bW_{\rmc}^{\T}\bX_0, \, \bhV_0 = \bW_{\rmc}^{\T}\bV_0.
		\end{align*}
	\end{algorithmic}
\end{algorithm}

\section{Error bounds}\label{sec:ErrEst}
In this section we develop a posteriori error bounds for the methods of this article.
Therefore, we use the fact that
\[ \|\by\|_{L_2} = \|\bh \ast \bu\|_{L_2} = \|\bH  \bU\|_{\cH_2} \leq \|\bH\|_{\cH_2}\|\bU\|_{\cH_{\infty}}  \]
where $\bU = \cL\{\bu\}$ and $\bh(t):=\boldsymbol{\cC}\exp(\boldsymbol{\cE}^{-1}\boldsymbol{\cA} t)\boldsymbol{\cE}^{-1}\boldsymbol{\cB}$.

Firstly, we use the above inequality to find  a posteriori error bounds for the reduction scheme presented in Subsection \ref{sec:sep_MOR}. As a consequence, a possible error bound for the reduced subsystems approximation is 
\begin{align*}
	\|\by- \bhy \|_{L_{2}} \leq \|\bH_{\SO} &-\bhH_{\SO}\|_{\cH_{2}}\|\cL(\bu)\|_{\cH_{\infty}}
	\\
	&+\|\bH_{\Xinit}-\bhH_{\Xinit}\|_{\cH_2}\|\bz_0\|_2 \\
	&\qquad+\|\bH_{\Vinit}-\bhH_{\Vinit}\|_{\cH_2}\|\bw_0\|_2.
\end{align*}
Using the $\cH_2$ norm has the advantage of less computational costs. However, one needs to have $\bu \in \cH_{\infty}$, which applies some restrictions to the family of inputs $\bu$ because $\bu\in\cH_{\infty}$ is a stronger condition than $ \bu\in\cH_2$.

We compute the $\cH_2$ norm of the difference between the transfer function $\bH(s) := \boldsymbol{\cC}\left( \boldsymbol{\cA} -s\boldsymbol{\cE} \right)^{-1}\boldsymbol{\cB}$ and the reduced transfer function $\hat{\bH}(s) := \boldsymbol{\hat{\cC}}\left( \boldsymbol{\hat{\cA}} -s\boldsymbol{\hat{\cE}} \right)^{-1}\boldsymbol{\hat{\cB}}$ in the following way

\begin{align}\label{eq:H2computation}
&\|\bH - \hat{\bH}\|_{\cH_2}^2\nonumber\\
		&=\int_{0}^{\infty}\trace{\left(\bH(s)-\hat{\bH}(s)\right)^{\HH}\left(\bH(s)-\hat{\bH}(s)\right)}\mathrm{d}s\nonumber\\
		&=\int_{0}^{\infty}\trace{\bH(s)^{\HH}\bH(s)}\mathrm{d}s 
		-2\int_{0}^{\infty}\trace{\bH(s)^{\HH}\hat{\bH}(s)}\mathrm{d}s\nonumber \\ 
		&\qquad\qquad\qquad\qquad\quad\qquad\quad+ \int_{0}^{\infty}\trace{\hat{\bH}(s)^{\HH}\hat{\bH}(s)}\mathrm{d}s\nonumber\\
		&=\trace{\boldsymbol{\cC}\boldsymbol{\cP}\boldsymbol{\cC}^{\T}} -2\trace{\boldsymbol{\cC}\boldsymbol{\tilde{\cP}}\boldsymbol{\hat{\cC}}^{\T}} + \trace{\boldsymbol{\hat{\cC}}\boldsymbol{\hat{\cP}}\boldsymbol{\hat{\cC}}^{\T}},
\end{align}

where $\boldsymbol{\cP}$ and $\boldsymbol{\hat{\cP}}$ are the controllability Gramians of the original first-order system and the reduced first-order system.
The cross Gramian $\boldsymbol{\tilde{\cP}}$ solves the Sylvester equation
\[
\boldsymbol{\cA}\boldsymbol{\tilde{\cP}}\boldsymbol{\hat{\cE}}^{\T} + \boldsymbol{\cE}\boldsymbol{\tilde{\cP}}\boldsymbol{\hat{\cA}}^{\T} = - \boldsymbol{\cB}\boldsymbol{\hat{\cB}}^{\T}.
\]
The controllability Gramians $\boldsymbol{\cP}$ of the full system needs to be computed anyway to apply balanced truncation and the reduced Gramians $\boldsymbol{\tilde{\cP}}$ and $\boldsymbol{\hat{\cP}}$ are cheap to compute.


The error estimation for the combined Gramian and the resulting reduced-order system is equal to the error estimate above, where the projection matrices that lead to the reduced system are the same for the three subsystems.
On the other hand one can evaluate 
\begin{align*}
	\|\by- \bhy \|_{L_{2}} 
	&\leq \|\bH-\bhH\|_{\cH_{2}}\|\cL(\bu_{\rmc})\|_{\cH_{\infty}}\\
	&=  \|\bH-\bhH\|_{\cH_{2}}\left(\|\cL(\bu)\|_{\cH_{\infty}}+\|\bz_0\|_2 + \|\bw_0\|_2\right)
\end{align*}
for the matrix $\boldsymbol{\cB}$ as in \eqref{eq:combB} and $\bu_{\rmc}:=\footnotesize
\left[\hspace*{-3pt} \begin{array}{c}
\bu(t)\\
\bz_0\delta(t)\\
\bw_0\delta(t) \end{array}\hspace*{-3pt} \right]
$.
This estimation can be computed using the equation \eqref{eq:H2computation}.

\section{Numerical results}\label{sec:NumRes}
In this section, we illustrate the procedure presented in this article using two different examples.
The first example is a vibrational model of a building and the second one a mass spring damper system. 
We evaluate and compare for each example three reduced systems.
We obtain the first one by applying balanced truncation to the full system \eqref{eq:SO_sys} that does not consider the initial conditions.
Hence, the projection matrices result from the evaluation of the Gramians of the homogeneous system \eqref{eq:SO_homo}.
The second reduced system is obtained by reducing the three subsystems, separately, as presented in this article.
The third method uses the combined Gramian presented in Section \ref{sec:combGram} to obtain the reduced-order model.
For all reduced systems, we evaluate the output behavior and the corresponding output error.

The computations were done on a computer with 4 Intel\textsuperscript{\textregistered} Core\texttrademark  i5-4690 CPUs running at 3.5 GHz.
The experiments use \matlab R2017a.
In the second example, the Lyapunov equations are solved using the methods from the M-M.E.S.S. toolbox \cite{Saak2021}.

We will refer to the original system \eqref{eq:SO_sys} as \texttt{FOM}, in the following, and to the reduced system generated by standard balanced truncation that considers homogeneous systems, i.e, by applying second-order balanced truncation as described in Subsection \ref{sec:BT_SO} by \texttt{ROM{\_}HOM}. The reduced system approximation that is obtained by applying method 1 introduced in Subsection \ref{sec:sep_MOR} is referred to as \texttt{ROM{\_}SPL} and the reduced system that is generated by applying method 2 introduced in Subsection \ref{sec:combGram} as \texttt{ROM{\_}COM}.

\subsection{Building example}
In this section we consider the building example from page 17 of the technical report \cite{morAntSG01}.
The data are available in \cite{SLCbasic}.
The dimension of the matrices are $n=24,\, m = p = 1$.
For the projection matrix $\bW_{\SO}$ that results from the balanced truncation procedure for the homogeneous second-order system \eqref{eq:SO_homo} we consider the singular value decomposition 
\[
\bU\boldsymbol{\Sigma} \bX^{\T} = \bW_{\SO}.
\]
Assume that $\mathrm{rank}(\bW)=\ell$. The position and velocity initial condition are the $(\ell +1)$-st column of $\bU$:
\[
\bX_0 = \Xinit = \bV_0 = \bdx_0 = \bU[\, : \, , \, \ell+1 \, ].
\]
In this example, the separately reduced systems and the combined reduced system are truncated with a reduced dimension $r=10$.
Figure \ref{pic:build_output} shows the output behavior of the original system and the reduced ones for an input $\bu(t) = 0.2\cdot e^{-t}$. 
We observe that the original output behavior that is depicted in green is well approximated by the separately reduced system \texttt{ROM{\_}SPL} that is depicted by the blue, dashed line.
The reduced system \texttt{ROM{\_}COM} using the combined Gramian (depicted by the orange colored, dashed line) provides a proper approximation of the original output as well.
Additionally, we see that the reduced output of the reduced system \texttt{ROM{\_}HOM}, which is depicted in red, fails in approximating the original system's transient behavior.

Figure \ref{pic:build_error} depicts the errors and their $\ell_2$-norms.
The light blue line with markers depicts the error of the separately reduced system \texttt{ROM{\_}SPL} and the dashed, brown colored line the error of the reduced system \texttt{ROM{\_}COM} using the combined Gramian.
The reduced system \texttt{ROM{\_}HOM} leads to the error depicted by the dashed, orange colored line.
We observe, that the separately reduced system and the reduced system that uses the combined Gramain lead to errors that are significantly smaller than the error corresponding to the reduced system \texttt{ROM{\_}HOM}.
Additionally, we evaluate the actual $\ell_2$-norm error.
Therefore, we plot the integral
\begin{equation}\label{eq:l2_integral}
	\sqrt{\int_{0}^{t}\|\by(t) - \bhy(t)\|^2 \mathrm{d}t}
\end{equation}
that converges to the $\ell_2$-norm of the error.
The dark blue, dashed line with markers is the integral \eqref{eq:l2_integral} converging to the actual $\ell_2$-norm error of the separately reduced system \texttt{ROM{\_}SPL}.
The error bound from Section \ref{sec:ErrEst} provides a value of $3.2740\cdot 10^{-5}$ (depicted by the black line).
We see that this error estimator provides a proper upper bound of the actual $\ell_2$-norm error.
The green line with markers provides the integral \eqref{eq:l2_integral} corresponding to the combined Gramian reduced system \texttt{ROM{\_}COM} and its error estimation $1.5469\cdot 10^{-4}$ is depicted by the dashed, black line.
The red line shows the integral \eqref{eq:l2_integral} of the reduced system \texttt{ROM{\_}HOM}.
It confirms again, that this method fails for this example.
\begin{figure}\label{pic:build}
	\centering
	\newlength\fheight 
	\newlength\fwidth 
	\setlength{\fwidth}{0.9\textwidth}
	\setlength{\fheight}{3.5cm}
	\subfloat[Build - Output]{%
		\resizebox*{7cm}{!}{\input{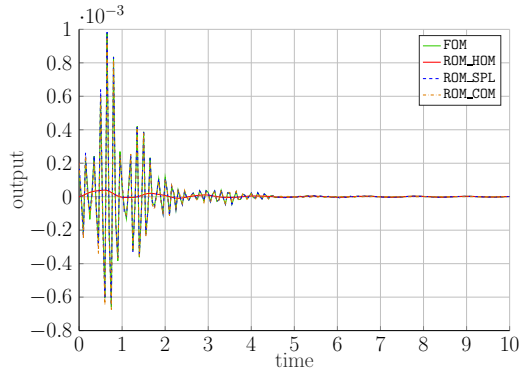} \label{pic:build_output}}}\vspace{20pt}
	\subfloat[Build - Error]{%
		\resizebox*{7cm}{!}{\input{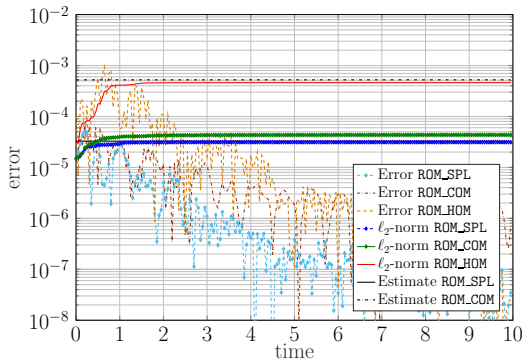}\label{pic:build_error}}}
	\caption{Build example}
\end{figure}

\subsection{Mass spring damper example}
The mass spring damper model we consider in this section is presented in \cite{morWiki_msd}.
More detailed background can be found in \cite{morKawS15}.

We choose the model of dimensions $n=2000$, $m=p=1$.
The input is the external forcing on the $n$-th mass and the output observes the $n$-th mass.

The initial conditions are set to be the last and the first unit vector 
\[
\bX_0 = \Xinit := e_n, \qquad \bV_0 = \dot{\bx}_0 := e_1.
\]

In this example, we truncate the systems with a tolerance of $10^{-4}$, i.e. all Hankel singular values smaller than $10^{-4}\cdot\sigma_1$ are truncated.
That way, we obtain reduced systems of dimensions $147,\, 180,\, 98$ of the three systems resulting from the superposition method and the dimension $157$ for the system reduced using the combined Gramians.

Figure \ref{pic:msd_output} shows the output behavior of the systems for the input $\bu(t) = 0.2\cdot e^{-t}$.
The output behavior of the original system is depicted in green.
The blue, dashed line displays the output composed by the separately reduced systems \texttt{ROM{\_}SPL} and the orange colored, dashed line the reduced system \texttt{ROM{\_}COM} using the combined Gramian.
The reduced output resulting from the reduced system \texttt{ROM{\_}HOM} is depicted in red.
We observe that all outputs approximate the original system behavior. although \texttt{ROM{\_}HOM} shows oscillations of slightly higher magnitude than the \texttt{FOM} for some time.

The output errors and their $\ell_2$-norms are illustrated in Figure \ref{pic:msd_error}.
The light blue line with markers, the brown colored, dashed line and the orange colored, dashed line show the error of the separately reduced outputs, the output corresponding to the combined Gramian and the output resulting form the reduced system \texttt{ROM{\_}HOM}, respectively.
We observe again that the separately reduced system \texttt{ROM{\_}SPL} and the reduced system \texttt{ROM{\_}COM} using the combined Gramian lead to smaller errors.
Additionally, we evaluate the actual $\ell_2$-norm error and plot the integral \eqref{eq:l2_integral}
that converges to the $\ell_2$-norm of the error.
The dark blue, dashed line with markers shows the integral \eqref{eq:l2_integral} for the separately reduced system \texttt{ROM{\_}SPL} and the green one the integral for the reduced system \texttt{ROM{\_}COM} using the combined Gramian.
The error estimator from Section \ref{sec:ErrEst} provides $\ell_2$ error estimation values of $7.5490\cdot 10^{-3}$ and $3.1922\cdot 10^{-2}$ for this example.
It is depicted in Figure \ref{pic:msd_error} by the black and black, dashed lines.
We observe that the error estimations are conservative.
The integral \eqref{eq:l2_integral} of the reduced system \texttt{ROM{\_}HOM} is depicted in red. It converges to a $\ell_2$ error that is larger than for the first two reduction methods.

\begin{figure}\label{pic:msd}
	\centering
	\subfloat[Mass-spring-damping - Output]{%
		\resizebox*{6cm}{!}{\input{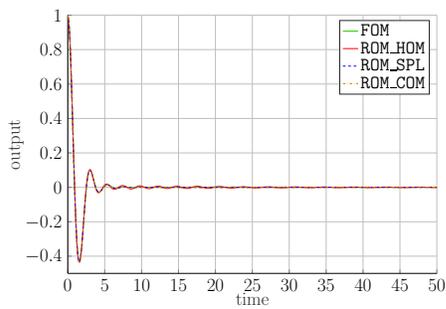} \label{pic:msd_output}}}\vspace{20pt}
	\subfloat[Mass-spring-damping - Error]{%
		\resizebox*{6cm}{!}{\input{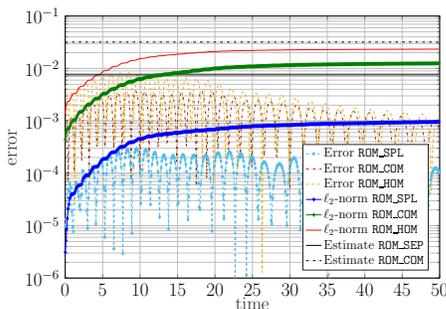}\label{pic:msd_error}}}
	\caption{Mass-spring-damping example}
\end{figure}

\section{Conclusion}
We have proposed two approaches for constructing a reduction of second-order linear time-invariant systems with inhomogeneous initial conditions.
First, we have used a superposition of the output into the input to output mapping, the state initial condition to output mapping and the velocity initial condition to output mapping.
The three subsystems have been reduced, separately, such that the original system can be approximated well.
Afterward, a combined Gramian has been used to derive projection matrices that reduce the system, including the initial conditions, all at once.
For those reduction processes we have suggested new Gramians for inhomogeneous second-order systems.


\addcontentsline{toc}{section}{References}
\bibliographystyle{plainurl}
\bibliography{References,mor,csc}
  
\end{document}